\documentclass[12pt,leqno,oneside]{amsart}
\usepackage{mathrsfs,dsfont}
\usepackage{amsmath,amstext,amsthm,amssymb,bbm}
\usepackage{comment}
\usepackage{charter}
\usepackage{typearea}
\usepackage{pdfsync}
\usepackage[width=6.5in,height=8.8in]{geometry}

\def\bt{\begin{thm}}
\def\et{\end{thm}}
\def\bl{\begin{lem}}
\def\el{\end{lem}}
\def\bd{\begin{defi}}
\def\ed{\end{defi}}
\def\bc{\begin{cor}}
\def\ec{\end{cor}}
\def\bp{\begin{proof}}
\def\ep{\end{proof}}
\def\br{\begin{rem}}
\def\er{\end{rem}}

\newtheorem{thm}{Theorem}[section]
\newtheorem{prop}[thm]{Proposition}
\newtheorem{lem}[thm]{Lemma}
\newtheorem{defn}[thm]{Definition}
\newtheorem{example}[thm]{Example}
\newtheorem{rem}[thm]{Remark}
\newtheorem{cor}[thm]{Corollary}

\numberwithin{equation}{section}

\newcommand{\C}{\Bbb{C}^m}
\newcommand{\pn}{\mathcal{P}_n}

\newcommand{\norm}{\|\varphi_n\|_{\mathscr{C}^3,B_R}}
\newcommand{\bthm}{\begin{thm}}
\newcommand{\ethm}{\end{thm}}
\newcommand{\bstp}{\begin{stp}}
\newcommand{\estp}{\end{stp}}
\newcommand{\blemma}{\begin{lemma}}
\newcommand{\elemma}{\end{lemma}}
\newcommand{\bprop}{\begin{prop}}
\newcommand{\eprop}{\end{prop}}
\newcommand{\bpf}{\begin{pf}}
\newcommand{\epf}{\end{pf}}
\newcommand{\bdefn}{\begin{defn}}
\newcommand{\edefn}{\end{defn}}
\newcommand{\brk}{\begin{rmrk}}
\newcommand{\erk}{\end{rmrk}}
\newcommand{\bcrl}{\begin{crl}}
\newcommand{\ecrl}{\end{crl}}

\title[Asymptotic normality of random zeros]{Asymptotic normality of linear statistics of\\ zeros of random polynomials}
\author{Turgay Bayraktar}
\date{\today}


\address{Mathematics Department, Syracuse University, NY, USA}
\email{tbayrakt@syr.edu}

\begin{document}

\begin{abstract}
In this note, we prove a central limit theorem for smooth linear statistics of zeros of random polynomials which are linear combinations of orthogonal polynomials with iid standard complex Gaussian coefficients. Along the way, we obtain Bergman kernel asymptotics for weighted $L^2$-space of polynomials endowed with varying measures of the form $e^{-2n\varphi_n(z)}dz$ under suitable assumptions on the weight functions $\varphi_n$.
\end{abstract}

\maketitle
%\tableofcontents
\section{Introduction}

Let $\varphi_n:\C\to \Bbb{R}$ be a sequence of weight functions satisfying 
\begin{equation}\label{growth}
\varphi_n(z)\geq (1+\epsilon)\log|z|\ \text{for}\ z\in \C\setminus B_R
\end{equation} where $\epsilon>0$ and $B_R:=\{z\in\C:|z|\leq R\}$ for some $R\gg1.$ 
We define a norm on the space $\mathcal{P}_n$ of polynomials of degree at most $n$ by 
\begin{equation}\label{n}
\|f_n\|_{\varphi_n}^2:=\int_{\C}|f_n(z)|^2e^{-2n\varphi_n}dV_m(z)
\end{equation} where $dV_m$ denotes the Lebesgue measure on $\C$. For a fixed ONB $\{P_j^n\}_{j=1}^{d_n}$ of $\mathcal{P}_n$ with respect to the norm (\ref{n}) we consider \textit{Gaussian random polynomials}
$$f_n(z)=\sum_{j=1}^{d_n}c_jP_j^n(z)$$ where $c_j$ are iid complex Gaussian random variables of mean zero and variance one and $d_n:=\dim(\mathcal{P}_n)$. We endow $\mathcal{P}_n$ with the $d_n$-fold product measure. Clearly, the induced Gaussian measure on $\mathcal{P}_n$ is independent of the choice of the ONB.

\par In this note, we consider asymptotic normality of zeros of random polynomials 
$f_n(z).$ Namely, for a fixed smooth test form $\Phi\in \mathcal{D}^{m-1,m-1}(\C)$ supported in the bulk $\mathcal{B}_n$ (see \ref{bulk} for its definition) we consider the random variables
$$\mathcal{Z}_n^{\Phi}(f_n):= \langle [Z_{f_n}],\Phi\rangle $$ where $[Z_{f_n}]$ denotes the current of integration along the zero divisor of $f_n.$  By choosing $R$ large enough we assume that $supp(\Phi)\subset B_R.$ The random variables $\mathcal{Z}_{f_n}$ are often called \textit{linear statistics} of zeros. A form of universality for zeros of random polynomials is central limit theorem (CLT) for linear statistics of zeros, that is 
$$\frac{\mathcal{Z}_n^{\Phi}-\Bbb{E} \mathcal{Z}_n^{\Phi}}{\sqrt{Var[\mathcal{Z}_n^{\Phi}]}}$$
converge in distribution to the (real) Gaussian random variable $\mathcal{N}(0,1)$ as $n\to \infty$. 
Here, $\Bbb{E}$ denotes the expected distribution and $Var$ denotes the variance of the random variable $\mathcal{Z}^{\Phi}_n.$
\par In complex dimension one, Sodin and Tsirelson \cite{STr} obtained
asymptotic normality of $Z_n^{\psi}$ for Gaussian analytic functions and a $\mathscr{C}^2$ function $\psi$ by using diagram technique which is a classical approach in probability theory to prove CLT for Gaussian random processes whose variances are asymptotic to zero. More precisely, they  observed that asymptotic normality of linear statistics reduce to Bergman kernel asymptotics (Theorem \ref{ST}).  See also \cite{NaSo2} for a generalization of this result to the case where $\psi$ is merely a bounded function by using a different method. On the other hand, Shiffman and Zelditch \cite{SZ3} pursued the idea of Sodin and Tsirelson and generalized their result to the setting of random holomorphic sections for a positive line bundle $L\to X$ defined over a projective manifold. Building upon ideas from \cite{STr,SZ3} and using the Bergman kernel asymptotics obtained in \S2 we prove a CLT for linear statistics under suitable assumptions on the weight functions $\varphi_n:$ %of the curvature forms of $\varphi_n$ and the sup-norm of the third order derivatives of $\varphi_n$ on support of $\partial\overline{\partial}\Phi.$
\begin{thm}\label{CLT} Let $\Phi$ be a real $(m-1, m-1)$ form with $\mathscr{C}^3$ coefficients such that $\partial\overline{\partial}\Phi\not\equiv0$ and supported in an open set $U$ such that $U\Subset \mathcal{B}_n$ for every $n.$ Let $\varphi_n$ be a sequence of $\mathscr{C}^3$ weight functions satisfying (\ref{growth}) for each $n$ and fixed $R\gg1.$ We assume that
\begin{itemize}
 \item[(1)] $\sup_n\max_{z\in B_R}|\varphi_n(z)|<\infty$ and $\sup_n\max_{z\in B_R}\|d\varphi_n(z)\|<\infty.$
\item[(2)] $\norm=o(\frac{\sqrt{n}} {\log^3 n})$ as $n\to \infty.$
\item[(3)] there exist constants $A,c>0$ such that
$$ci\partial\overline{\partial} |z|^2\leq i\partial\overline{\partial} \varphi_n\leq Ai\partial\overline{\partial} |z|^2 $$
 on $ \overline{U}.$
\end{itemize}
Then linear statistics
$$\frac{\mathcal{Z}_n^{\Phi}-\Bbb{E} \mathcal{Z}_n^{\Phi}}{\sqrt{Var[\mathcal{Z}_n^{\Phi}]}}$$ converge in distribution to the (real) Gaussian $\mathcal{N}(0,1)$ as $n\to \infty.$
\end{thm}
Theorem \ref{CLT} generalizes the results of \cite{STr,SZ3} to the present setting. We remark that in the present setting the zero currents themselves may not converge almost surely to a deterministic current. It follows from \cite[\S 6]{BloomL} (see also \cite{BloomS} for unweighted case and \cite{BloomL,B6,B7} for more general distributions) that in the presence of a single weight  (i.e. $\varphi=\varphi_n$ for all $n$), normalized zero currents $\frac1n[Z_{f_n}]$ converge almost surely to equilibrium current  $\frac{i}{\pi}\partial\overline{\partial}\varphi^{e}$ in the sense of currents on $\C$ (see \ref{ex} for definition of $\varphi^{e})$. The latter current coincides with the weighted equilibrium measure (cf. \cite{SaffTotik}) of the compact set defined by (\ref{sup}) in complex dimension one. In this special case, our methods allows us to obtain a CLT for a $\mathscr{C}^2$ weight function in the bulk thanks to the Bergman kernel asymptotics (see Remark \ref{rem}). 
\begin{thm}\label{CLT2}
Let $\varphi_n=\varphi:\C \to\Bbb{R}$ be a $\mathscr{C}^2$ weight function satisfying (\ref{growth}) and $\Phi$ be a real $(m-1, m-1)$ form with $\mathscr{C}^3$ coefficients such that $\partial\overline{\partial}\Phi\not\equiv0$ and supported in the interior of the Bulk $\mathcal{B}.$ Then linear statistics
$$\frac{\mathcal{Z}_n^{\Phi}-\Bbb{E} \mathcal{Z}_n^{\Phi}}{\sqrt{Var[\mathcal{Z}_n^{\Phi}]}}$$ converge in distribution to the (real) Gaussian $\mathcal{N}(0,1)$ as $n\to \infty.$
\end{thm}
\par A generalization of Theorem \ref{CLT} to the line bundle setting will appear in \cite{BCM}.  In the latter setting, the role of $\varphi_n$ is played by a metric. In this general geometric setting, a Johansson type large deviations theorem is obtain by Dinh and Sibony \cite[Corollary 7.4]{DS3} which is generalized by Dinh-Marinescu-Schmidt \cite{DMS1} and Coman-Marinescu-Nguyen \cite{CMN}. These large deviations theorems hold in a very general setting (see e.g. \cite[Theorem 4]{DMS1}). It is enough to have a Hilbert structure on the dual of the space of holomorphic sections (in particular, one doesn't need the metric to be smooth nor positive, we can apply it to singular metrics). %These results imply that $\mathcal{Z}_n^{\Phi}-\Bbb{E} \mathcal{Z}_n^{\Phi}$ converges almost surely to zero with a good speed estimate. 
We also remark that in the present setting due to the growth condition (\ref{growth}) the function $\varphi_n$ does not give rise to a metric on the hyperplane bundle $\mathcal{O}(1)\to \Bbb{CP}^m.$ 

%\section*{Acknowledgement}
I would like to thank the referee whose suggestions improved the exposition of the paper. 

\section{Bergman kernel asymptotics}
For a continuous weight function $\varphi_n$ satisfying (\ref{growth}) we denote the corresponding \textit{equilibrium potential}
\begin{equation}\label{ex}
\varphi^e_n(z):=\sup\{\psi(z):\psi\in \mathcal{L}(\C), \psi\leq \varphi_n\ \text{on}\  \C\}
\end{equation}
where $\mathcal{L}(\C)$ denote the Lelong class of psh functions. It is well known (cf. \cite[Appendix B]{SaffTotik}) that $\varphi_n^e\in \mathcal{L}_+(\C).$ Hence, it follows from Bedford-Taylor theory that its Monge-Amp\`ere $$MA(\varphi_n^e):=(dd^c\varphi_n^e)^m$$ is well-defined and does not put any mass on pluri-polar sets where $d=\partial+\overline{\partial}$ and $d^c:=\frac{i}{2\pi}(\overline{\partial}-\partial)$ so that $dd^c=\frac{i}{\pi}\partial\overline{\partial}$. Moreover, by  \cite[Appendix B]{SaffTotik} the measure $MA(\varphi_n^e)$ is supported on the set 
\begin{equation}\label{sup}
D_n:=\{z\in\C:\varphi_n^e(z)=\varphi_n(z)\}.
\end{equation}
Note that the set $D_n$ is closed (since $\varphi_n^e$ is usc) and bounded (by (\ref{growth})) and hence compact. By choosing $R$ large enough we may assume that $D_n\subset B_R.$ Furthermore, Berman \cite{Berman} proved that if $\varphi_n$ is $\mathscr{C}^{1,1}$ then
\begin{equation*}
MA(\varphi_n^e)=1_{D_n}\det(dd^c\varphi_n)dV_m(z)
\end{equation*}  where the latter $L^{\infty}_{loc} (m,m)$-form is obtained by point-wise calculation. 

\par The classical example $\varphi(z)=\frac{|z|^2}{2}$ gives $MA(\varphi^e(z))=\mathbbm{1}_{B}dV_m(z)$ where $\mathbbm{1}_{B}$ denotes the characteristic function of the unit ball in $\C.$

\par For a fixed orthonormal basis $\{P^n_j\}_{j=1}^{d_n}$ for $\pn$ with respect to the norm (\ref{n}) the Bergman kernel is given by
$$K_n(z,w):=\sum_{j=1}^{d_n}P_j^n(z)\overline{P_j^n(w)}$$ where $d_n=\dim(\mathcal{P}_n).$
We also denote the \textit{Bergman function} by
$$B_n(z):=K_n(z,z)e^{-2n\varphi_n(z)}=\sum_{j=1}^{d_n}|P^n_j(z)|^2e^{-2n\varphi_n(z)}.$$
Bergman function $B_n$ has the extremal property
$$B_n(z)=\sup_{f\in \pn}\frac{|f_n(z)|^2e^{-2n\varphi_n(z)}}{\|f_n\|_{\varphi_n}^2}.$$ Moreover, we have the following dimensional density property
$$\int_{\C}B_n(z)dV_m(z)=\dim(\pn)=O(n^m).$$

For a $\mathscr{C}^{1,1}$ weight function $\varphi_n$ satisfying the growth condition (\ref{growth}) we define its \textit{bulk} $\mathcal{B}_n$ by
\begin{equation}\label{bulk}
\mathcal{B}_n=\{z\in Int(D_n): \varphi_n\ \text{is smooth near}\ z\ \text{and}\ dd^c\varphi_n(z)>0 \}.
\end{equation} 
 For each compact set $K\subset \C$ we denote the $\mathscr{C}^k$ semi-norm of $\varphi_n$ on $K$ by $$\|\varphi_n\|_{\mathscr{C}^k,K}:=\sup_{z\in K}\{|D_z^{\alpha}\varphi_n(z)|:|\alpha|=k\}.$$
 %where differential operators $D^{\alpha}$ correspond to real coordinates. %and $\|\cdot\|_{\mathscr{C}^0,K}$ denotes the sup-norm on $K$.

\par By adapting the methods of Berman \cite{Berman} and Berndtsson \cite{Berndtsson2}, we obtain the first order asymptotics of the Bergman kernel  which coincide with the Tian-Zelditch-Catlin expansion for a positive Hermitian holomorphic line bundle (see also \cite{CMM,BCM} for an extension of this result to the line bundle setting).

\begin{thm}\label{thm1}
Let $K$ be a compact set contained in an open set $U$ such that $U\Subset \mathcal{B}_n$ for every $n.$ Let $\varphi_n$ be a sequence of $\mathscr{C}^3$ weight functions satisfying (\ref{growth}) for each $n$ and fixed $R\gg1.$ We assume that
\begin{itemize}
 \item[(1)]  $\sup_n\max_{z\in B_R}|\varphi_n(z)|<\infty$ and $\sup_n\max_{z\in B_R}\|d\varphi_n(z)\|<\infty.$
\item[(2)] $\norm=o(\frac{\sqrt{n}} {\log^3 n})$ as $n\to \infty.$
\item[(3)] there exist constants $A,c>0$ such that
$$ci\partial\overline{\partial} |z|^2\leq i\partial\overline{\partial} \varphi_n\leq Ai\partial\overline{\partial} |z|^2 $$
 on $ \overline{U}.$

\end{itemize}
 Then
\begin{equation}
\frac{n^{-m}K_n(z,z)e^{-2n\varphi_n(z)}}{det(dd^c\varphi_n(z))}\to 1 
\end{equation} uniformly on $K$ as $n\to \infty.$
%Moreover,
%\begin{equation}
%\frac{\pi^mn^{-m}K_n(z+\frac{x}{\sqrt{n}},z+\frac{y}{\sqrt{n}})}{\det(dd^c(\varphi_n^e(z)))e^{\langle \det(dd^c\varphi_n(z)) x,y\rangle}}\to 1
%\end{equation}
% uniformly on compact subsets of $\C_x\times \C_y.$
\end{thm}
 
 \begin{proof}
By choosing $R$ large enough we assume that $\overline{U}\subset B_R.$ We fix $z\in K,$ by translation we may assume that $z=0.$ Let $g_n$ be a holomorphic polynomial of degree at most 1 such that 
$$\varphi_n(\eta)=Re(g_n(\eta))+\sum_{j,k}a^n_{jk}\eta_j\overline{\eta}_k+\varphi_{n,3}(\eta).$$  %By assumption (1) polynomials $g_n$ are uniformly bounded on $K\subset B_R.$ 
  Furthermore, applying a unitary change of coordinates we define
 \begin{equation}\label{cord}\tilde{\varphi}_n(\zeta):=\varphi_n(\zeta)-Re(g_n(\zeta))=\sum_{j=1}^m \lambda_j^n|\zeta_j|^2+O(|\zeta|^3)
 \end{equation} where $\lambda_j^n$ are eigenvalues of $dd^c\varphi_n(0).$ We refer to $\zeta$ coordinates introduced in (\ref{cord}) as \textit{normal coordinates} centered at $z$.
 %Note that $\tilde{\varphi}_n$ also satisfies (1)-(3) (where $\overline{U}$ is replaced with its image under the corresponding unitary map) thanks to invariance of $B_R$ under unitary change of coordinates.
 
  Next, we fix $f_n\in \pn$ such that $\|f_n\|_{\varphi_n}=1$ then by sub-mean value inequality applied to the holomorphic function $f_ne^{-ng_n}$, using assumption (2) and applying a change of variables $w=\sqrt{n}\zeta$ we obtain
\begin{eqnarray}
|f_n(0)|^2e^{-2n\varphi_n(0)}=|f_n(0)e^{-ng_n(0)}|^2 &\leq& \frac{1}{\int_{\Delta^m(0,\frac{\log n}{\sqrt{n}})}e^{-2n\tilde{\varphi}_n(\zeta)}dV_m(\zeta)} \nonumber \\
&\leq&  \frac{1}{n^{-m}\int_{\Delta^m(0,\log n)}e^{-2n\tilde{\varphi}_n(\frac{w}{\sqrt{n}})}dV_m(w)} \nonumber \\
&\leq & \frac{\exp(\epsilon_{n})}{n^{-m}\int_{\Delta^m(0,\log n)}e^{-2\sum_{j=1}^m\lambda_j^n|w_j|^2}dV_m(w)} \nonumber\\
&\leq& \frac{n^m\exp(\epsilon_{n})}{(E[\sqrt{c}\log n])^m}\lambda_1^n\dots \lambda^n_m \label{ub}
\end{eqnarray}
where $\epsilon_n:=2\norm \frac{(\log n)^3}{\sqrt{n}}\to 0$ and we used $\lambda_j^n\geq c>0$ together with 
\begin{eqnarray*}
(E[\sqrt{c}\log n])^m &:=& \prod_{j=1}^m\int_{|\zeta_j|\leq \sqrt{c}\log n}e^{-2|\zeta_j|^2}d\zeta_j=(\frac{\pi}{2})^m(1-\exp(-2c(\log n)^2))^m\\
&\leq& \lambda_1^n\dots \lambda_m^n \int_{\zeta\in \Delta^m(0,\log n)}e^{-2\sum_{j=1}^{m}\lambda_j^n|\zeta_j|^2}dV_m.
\end{eqnarray*}
Hence, using the extremal property of $B_n(z)$ we obtain 
\begin{equation}\label{gub}
n^{-m}B_n(0)\leq \frac{\exp(\epsilon_{n})}{(E[\sqrt{c}\log n])^m}\lambda_1^n\dots \lambda^n_m.
\end{equation}
 
 In order to get the lower bound we use a generalization of H\"ormander's  $L^2$-estimates for the solution of $\overline{\partial}$-equation due to Demailly \cite[Theorem 5.1]{Dem82} (see also \cite[Theorem 5.4]{Berman}): %Here we use the fact that $dd^c\varphi_n$ defines a positive form with locally bounded coefficients. 
 Let $\chi$ be a smooth cut-off function such that $0\leq \chi\leq 1,$ it is supported in a polydisc $\Delta^m(0,2\delta)\subset U$ and $\chi\equiv 1$ on $\Delta^m(0,\delta).$ Define $\chi_n(z):=\chi(\frac{z}{r_n})$ where $r_n:=\frac{\log n}{\sqrt{n}}$ then by \cite[Theorem 5.4]{Berman}( see also \cite[Theorem 5.1]{Dem82}) applied on the line bundle $\mathcal{O}(n)\to \Bbb{P}^m$ with the weight function $$\psi_n(z):=(n-C)\varphi^e_n(z)+\frac{C}{2}\log(1+|z|^2)$$ and for $C>\frac{m+1}{1+2\epsilon},$ there exists a smooth function $u_n$ such that 
 $\overline{\partial}u_n=\overline{\partial}\chi_n$ and 
 \begin{equation}
 \int_{\C}|u_n|^2e^{-2\psi_n(\zeta)}\frac{1}{(1+|\zeta|^2)^{m+1}}dV_m \leq \frac{1}{C}\int_{\C}|\overline{\partial}\chi_n(\zeta)|^2 e^{-2\psi_n(\zeta)}\frac{1}{(1+|\zeta|^2)^{m+1}}dV_m(\zeta).
 \end{equation}
Then by using $\varphi^e_n\leq \varphi_n$ on $\C$ and assumption (1) we infer that
\begin{eqnarray}
 \int_{\C}|u_n|^2e^{-2\psi_n}\frac{1}{(1+|\zeta|^2)^{m+1}}dV_m &\geq& \int_{\C}|u_n|^2e^{-2n\varphi_n}\exp(2C\varphi_n-(C+m+1)\log(1+|\zeta|^2))dV_m\nonumber\\
&=&\int_{|z|\geq R}+ \int_{|z|<R}\nonumber\\
%&\geq & \int_{|z|\geq R} |u_n|^2e^{-2n\varphi_n}dV_m +\int_{|z|<R}|u_n|^2e^{-2n\varphi_n}\exp(2C\varphi_n-(C+m+1)\log(1+|\zeta|^2))dV_m \nonumber \\
 &\geq& C_1 \int_{\C}|u_n|^2e^{-2n\varphi_n}dV_m
 \end{eqnarray}
 where $C_1>0$ is independent of $n.$
 On the other hand, using $\varphi_n^e=\varphi_n$ on the set $U\subset B_R$ where $\chi$ is supported and by assumptions (1) and (2) we obtain 
 \begin{eqnarray}
     \int_{\Delta^m(0,2\delta r_n)}|\overline{\partial}\chi_n(\zeta)|^2 e^{-2\psi_n(\zeta)}\frac{1}{(1+|\zeta|^2)^{m+1}}dV_m(\zeta)
  &\leq &  \frac{C_2}{r^2_n} \int_{\delta r_n<|\zeta|<2\delta r_n}e^{-2\psi_n(\zeta)}\frac{1}{(1+|\zeta|^2)^{m+1}}dV_m(\zeta) \nonumber\\
  &\leq&  \frac{C_3}{r^2_n}  \int_{\delta r_n<|\zeta|<2\delta r_n}e^{-2n\tilde{\varphi}_n}dV_m\nonumber \\ %\exp(2C\varphi_n-(C+m+1)\log(1+|\zeta|^2)dV_m 
  &\leq & \frac{C_3}{r^2_n}  \int_{\delta r_n<|\zeta|<2\delta r_n} \exp(-2nc|\zeta|^2)e^{2n\norm |\zeta|^3}dV_m\nonumber\\
  &\leq& C_4\exp(-2nc\delta^2r_n^2)\exp(8\delta^3\epsilon_n)r_n^{2m-2}\nonumber\\
  &\leq& \frac{C_5(\log n)^{2m-2}}{n^{m+1}}.
 \end{eqnarray}
Hence,
\begin{equation}\label{nu}
\int_{\C}|u_n|^2e^{-2\varphi_n}dV_m\leq  \frac{C_6(\log n)^{2m-2}}{n^{m+1}}.\end{equation}
  \par Next, since $u_n$ is holomorphic in the polydisc $\Delta^m(0,\delta r_n),$ by (\ref{ub}) and (\ref{nu}) we have
 \begin{eqnarray}
|u_n(0)|^2 &\leq&  \frac{n^m\exp(\epsilon_{n})}{(E[\sqrt{c}\delta\log n])^m}\lambda_1^n\dots \lambda^n_m \|u_n\|_{\varphi_n}^2\\
&\leq& \frac{C_5\exp(\epsilon_{n})(\log n)^{2m-2}}{n(E[\sqrt{c}\log n])^m}\lambda_1^n\dots \lambda^n_m. \nonumber
 \end{eqnarray}

\par Now, we define $h_n:=\chi_n-u_n.$ Note that $h_n$ is holomorphic and $\|h_n\|_{\varphi_n}<\infty.$ Using the growth condition (\ref{growth}) one can show that $h_n$ is a polynomial of degree at most $n$ (see \cite[Lemma 5.5]{Berman}). Furthermore,
\begin{eqnarray}
|h_n(0)|^2e^{-2n\varphi_n(0)} &\geq& (1-|u_n(0)|)^2\\
&\geq& \big(1-C_7\frac{(\log n)^{m-1}}{\sqrt{n}})^2. \nonumber
\end{eqnarray}
 On the other hand,
 \begin{eqnarray}
 \|h_n\|_{\varphi_n}^2 &\leq& (\|\chi_n\|_{\varphi_n}+\|u_n\|_{\varphi_n})^2\\
 &\leq& \big[\big(\int_{\Delta^m(0,2\delta r_n)}e^{-2n\varphi_n}dV_m\big)^{\frac12}+\big(\int_{\C}|u_n|^2e^{-2n\varphi_n}dV_m\big)^{\frac12}]^2\nonumber\\
 &\leq& \big[\sqrt{(\frac{\pi}{2})^mn^{-m}\frac{\exp(\epsilon_{n})}{\lambda_1^n\dots\lambda^n_m}}+\sqrt{\frac{C_6(\log n)^{2m-2}}{n^{m+1}}}\big]^2\nonumber\\
 &\leq&(\frac{\pi}{2})^m\frac{n^{-m}\exp(\epsilon_{n})}{\lambda_1^n\dots \lambda^n_m}[1+\sqrt{\frac{C_8(\log n)^{2m-2}}{n}}]^2\nonumber
 \end{eqnarray}
thus, combining these estimates and using extremal property of $B_n$ we obtain
\begin{equation}\label{help}
B_n(0)\geq \frac{|f_n(0)|^2e^{-2n\varphi_n(0)}}{\|f_n\|_{\varphi_n}^2}\geq (\frac{2}{\pi})^m n^m \lambda_1^n\dots\lambda_m^n (1-C_9\frac{(\log n)^{m-1}}{\sqrt{n}})
\end{equation} where $C_j>0$ for $j=1,\dots ,9$ and does not depend on $n$.\\
Finally, using $\det(dd^c(\varphi_n(0))=(\frac{2}{\pi})^m\lambda_1^n\dots\lambda_m^n$ finishes the proof.
  \end{proof}
\begin{rem}\label{rem}
We stress that in the presence of a single $\mathscr{C}^2$ weight function (i.e. $\varphi=\varphi_n$ for all $n$) one can show that the estimates (\ref{gub}) and (\ref{help}) remains valid by using 
$$\sup_{|\zeta|\leq R}\big|n\varphi(\frac{\zeta}{\sqrt{n}})-\sum_{j=1}^n\lambda_j|\zeta_j|^2\big|\to 0$$
where $\zeta$ are coordinates centered at $z\in \mathcal{B}$ as in Theorem \ref{thm1} (see \cite[\S 3]{Berman} for details). In particular, assumptions (1), (2) and (3) are redundant in this special case. The same conclusion also applies to Theorems \ref{thm3} and \ref{thm2}.
\end{rem}
 Now, for fixed $z\in K$ which we identify with the origin as in Theorem \ref{thm1} and letting $\Lambda_n:=diag[\lambda_1^n,\dots,\lambda_m^n]$ whose diagonal entries are the eigenvalues of the curvature form $dd^c\varphi_n(z).$ Using normal coordinates centered at $z,$ it follows from Cauchy-Schwarz inequality, (\ref{gub}) and 
$$\langle u,v\rangle=\frac12(|u|^2+|v|^2-|u-v|^2)+iIm(\langle u,v\rangle) $$
that the holomorphic functions 
$$h_n(u,v):=\frac{n^{-m}K_n(n^{-\frac12}u,n^{-\frac12}\overline{v})\exp(-n(g_n(n^{-\frac12}u)+\overline{g_n(n^{-\frac12}\overline{v})})}{\det\Lambda_n\exp(2\langle \Lambda_n u,\overline{v}\rangle)}$$ are uniformly bounded on the set  $$\Omega:=\{(u,v):|u|,|v|\leq R\}$$ and passing to a subsequence $h_{n_k}$ we may assume that 
$h_{n_k}\to h_{\infty}$ locally uniformly. On the other hand, by Theorem \ref{thm1} we have $h_{\infty}=(\frac{\pi}{2})^m$ on the set $M=\{(u,\overline{u}):|u|\leq R\}.$ Since $M$ is maximally totally real and $h_{\infty}=(\frac{\pi}{2})^m$ on $M$ we conclude that $h_{\infty}\equiv(\frac{\pi}{2})^m$ on $\Omega$. Since this argument applies to every subsequence of $h_n$ we obtain the following near diagonal asymptotics:
   \begin{thm}\label{thm3} Let $K,\varphi_n$ be as in Theorem \ref{thm1} and $z\in K$ fixed. Then in normal coordinates centered at $z$ 
     \begin{equation}\label{near}
%\frac{n^{-m}K_n(n^{-\frac12}u,n^{-\frac12}v)\exp(-\frac12|\sqrt{\Lambda_n}u|^2-\frac12|\sqrt{\Lambda_n}v|^2)}{\det\Lambda_n\exp(\langle \Lambda_nu,v\rangle)}
\frac{n^{-m}K_n(n^{-\frac12}u,n^{-\frac12}\overline{v})\exp(-n(g_n(n^{-\frac12}u)+\overline{g_n(n^{-\frac12}\overline{v})})}{\det\Lambda_n\exp(2\langle \Lambda_n u,\overline{v}\rangle)}\to (\frac{\pi}{2})^m
  \end{equation}
  uniformly  on compact subsets of $\C_u\times \C_v$ as $n\to \infty.$
  \end{thm}
Off diagonal asymptotics of the Bergman kernel has been considered by various authors in different settings (eg. \cite{Christ,Lindholm,Delin,Berman}, see also \cite{BSZ,SZ2} for the line bundle setting). We adapt the methods of \cite{Lindholm,Delin} to our setting and obtain the following off-diagonal asymptotics of the Bergman kernel.
 \begin{thm}\label{thm2}
 Let $K,\varphi_n$ be as in Theorem \ref{thm1} then for every $z,w\in K$
\begin{equation} 
\frac{n^{-m}|K_n(z,w)|e^{-n\varphi_n(z)-n\varphi_n(w)}} {\sqrt{\det(dd^c\varphi^e_n(z))\det(dd^c\varphi^e_n(w))}}\leq Ce^{-T\sqrt{n}|z-w|}
\end{equation}
where $C,T>0$ are independent of $n$ (but depend on the lower bound of $dd^c\varphi_n).$
 \end{thm} 
  \begin{proof}
  \textbf{Case 1:} Assume that $|z-w|\leq\frac{8}{\sqrt{n}}.$ Consider  
  $$F_n(w):=\overline{K_n(z,w)}$$ which is holomorphic in $w.$ As in (\ref{ub}) by sub-mean value inequality we have 
  $$|F_n(w)|^2e^{-2n\varphi_n(w)}\leq \frac{n^m\exp(\epsilon_{n})}{E[c\log n]^m}(\frac{\pi}{2})^m\det(dd^c\varphi_n(w))\|F_n\|_{\varphi_n}^2$$
  where 
  \begin{eqnarray*} \|F_n\|_{\varphi_n}^2 &=&\int_{\C}|K_n(z,\zeta)|^2e^{-2\varphi_n}(\zeta)dV_m(\zeta)
  = K_n(z,z) \\ &=&\sup_{g\in \mathcal{P}_n, \|g\|_{\varphi_n}=1}|g(z)|^2\\
  &\leq& \frac{n^m\exp(\epsilon_{n})}{E[c\log n]^m}e^{2n\varphi_n(z)}(\frac{\pi}{2})^m\det(dd^c\varphi_n(z)).
  \end{eqnarray*}
  Combining these estimates we get,
  $$|K_n(z,w)|^2e^{-2n\varphi_n(w)}\leq \frac{n^{2m}\exp(2\epsilon_{n})}{E[c\log n]^{2m}}e^{2n\varphi_n(z)}(\frac{\pi}{2})^m\det(dd^c\varphi_n(z))(\frac{\pi}{2})^m\det(dd^c\varphi_n(w))$$
  hence,
  $$\frac{n^{-m}|K_n(z,w)|e^{-n\varphi_n(z)-n\varphi_n(w)}}{\sqrt{\det(dd^c\varphi_n(z))\det(dd^c\varphi_n(w))}}\leq C$$
  for some $C>0$ independent of $n$ and $z,w\in K.$\\
  
    \textbf{Case 2:} Now, we assume that $|z-w|>\frac{8}{\sqrt{n}}.$ By translation we may also assume that $w=0$ also by taking $n$ large we assume that $\Delta^m(z,\frac{1}{\sqrt{n}})$ is contained in the open set $U$ where $K\subset U\subset \mathcal{B}_n.$ Then as in case 1 we start with
    $$n^{-m}|F_n(0)|^2e^{-2n\varphi_n(0)}\leq \frac{\int_{\Delta^m(0,\frac{1}{\sqrt{n}})}|K_n(z,\zeta)|^2e^{-2n\varphi_n(\zeta)}dV_m(\zeta)}{n^m \int_{\Delta^m(0,\frac{1}{\sqrt{n}})}e^{-2n\tilde{\varphi}_n(\zeta)}dV_m(\zeta)} .$$
    The integral in the denominator can be estimated by $\exp(\epsilon_n)\det(dd^c\varphi_n(w)).$ In order to estimate the numerator let $\delta:=\frac{|z-w|}{2}>\frac{4}{\sqrt{n}}$ then
    $\Delta^m(0,\frac{1}{\sqrt{n}})\subset \{\zeta\in \C:|z-\zeta|>\delta\}$ and it is enough to estimate
    $$I_n(z):=\int_{|z-\zeta|>\delta}|K_n(z,\zeta)|^2e^{-2n\varphi_n(\zeta)}dV_m(\zeta).$$
    Let $\chi$ be a smooth non-negative function such that $\chi\equiv 1\ \text{on}\ \C\backslash B(z,\delta); \chi\equiv 0\ \text{on} \ B(z,\frac{\delta}{2})$ and $|\overline{\partial}\chi(\zeta)|^2\leq C_1\frac{\chi(\zeta)}{\delta^2}$ for some $C_1>0.$ Then
    \begin{eqnarray*}
    I_n(z) &\leq & \int_{\C}|K_n(z,\zeta)|^2\chi(\zeta)e^{-2n\varphi_n(\zeta)}dV_m(\zeta)\\
    &=& \sup_{g\in\mathcal{P}_n, \int_{\C}|g|^2\chi e^{-2n\varphi_n}dV_m=1} |\Pi_n(g\chi)(z)|^2
    \end{eqnarray*}
    where 
    $$\Pi_n(g\chi)(z):=\int_{\C}K_n(z,\zeta)g(\zeta)\chi(\zeta)e^{-2n\varphi_n(\zeta)}dV_m(\zeta)$$
  i.e. $\Pi_n$ denotes the Bergman projection form the weighted $L^2$ space of functions onto $\mathcal{P}_n$ defined by the kernel $K_n(z,w).$ Then writing 
  \begin{equation}\label{u}
  \Pi_n(g\chi)=g\chi-u
  \end{equation} we see that $\overline{\partial}(g\chi)=\overline{\partial}u$ and $u$ is orthogonal to $\mathcal{P}_n,$ that is $u$ is the solution of $\overline{\partial}$-equation with minimal $L^2$-norm. Moreover, since $\chi(z)=0$ we have 
    $$|\Pi(g\chi)(z)|^2=|u(z)|^2.$$
Now, using the fact that $u$ is holomorphic in $B(z,\frac{\delta}{2})$ and (\ref{ub}) we have
\begin{eqnarray*}
n^{-m}|u(z)|^2e^{-2n\varphi_n(z)} \leq \frac{\int_{\Delta^m(z,\frac{1}{\sqrt{n}})}|u(\zeta)|^2e^{-2n\varphi_n(\zeta)}dV_m(\zeta)}{n^m\int_{\Delta^m(z,\frac{1}{\sqrt{n}})}e^{-2n\tilde{\varphi}_n(\zeta)}dV_m(\zeta)}.
\end{eqnarray*}
where again the denominator can be estimated by $\exp(\epsilon_n)\det( dd^c(\varphi_n(z)).$ 

Next, we let $\Omega:=\frac{i}{2}\partial\overline{\partial}n\varphi_n$ and $\tau(\zeta):=\sqrt{n}Tdist(\zeta,B(z,\frac{\delta}{3}))$ where $T>\max(1,\sqrt{2c})$. Letting $\omega(\zeta):=e^{-\tau(\zeta)}$ by Leibniz rule we have
$$i\partial\overline{\partial}\omega=\omega i\partial \tau \wedge \overline{\partial}\tau\leq \frac{nT^2}{4}\omega i \partial \overline{\partial}|\zeta|^2$$
and using the assumption
$ci\partial\overline{\partial}|z|^2\leq i\partial\overline{\partial}\varphi_n(z)$ we get
$$i\partial\overline{\partial}\omega\leq \omega(i\partial\overline{\partial}n\varphi_n-\Omega).$$
Now, for every $g\in\mathcal{P}_n$ satisfying $\int|g|^2\chi e^{-2n\varphi_n}dV_m=~1$ and $u$ as in (\ref{u})  applying \cite[Theorem 4]{Berndtsson} we obtain
\begin{eqnarray*}
\int_{\Delta^m(z,\frac{1}{\sqrt{n}})}|u(\zeta)|^2e^{-2n\varphi_n(\zeta)}dV_m(\zeta) &\leq& \int_{B(z,\frac{\sqrt{m}}{\sqrt{n}})}|u(\zeta)|^2e^{-2n\varphi_n(\zeta)}\omega(\zeta)dV_m(\zeta)\\
&\leq&  \int_{B(z,\frac{\sqrt{m}}{\sqrt{n}})}|\overline{\partial}u|_{\Omega}^2e^{-2n\varphi_n(\zeta)}\omega(\zeta)dV_m(\zeta)\\
&\leq& \int_{\C\backslash B(z,\frac{\delta}{2})}|g(\zeta)\overline{\partial}\chi(\zeta)|^2_{\Omega}e^{-2n\varphi_n(\zeta)}\omega(\zeta)dV_m(\zeta)
\end{eqnarray*}
where $|\Phi|_{\Omega}$ denotes the norm of $(0,1)$ form $\Phi$ with respect to the metric defined by $\Omega$ which in turn can be estimated from above by Euclidean norm times $\frac{1}{cn}.$ Then the latter integral can be estimated by
\begin{eqnarray*}
&\leq& \frac{C_1}{\delta^2 cn}  \int_{\C\backslash B(z,\frac{\delta}{2})}|g(\zeta)|^2\chi(\zeta)e^{-2n\varphi_n(\zeta)}e^{-\frac{T\delta\sqrt{n}}{6}}dV_m(\zeta)\\
&\leq& \frac{C_1}{c\delta^2 n}e^{-\frac{T\delta\sqrt{n}}{6}}  \int_{\C}|g(\zeta)|^2\chi(\zeta)e^{-2n\varphi_n(\zeta)}dV_m(\zeta)\\
&=&  \frac{C_1}{c\delta^2n}e^{-\frac{T\delta\sqrt{n}}{6}}\\
&\leq& \frac{C_1}{16c} e^{-\frac{T\sqrt{n}|z-w|}{12}}.
\end{eqnarray*}
Finally, tracking back the inequalities we obtain
$$n^{-2m}|F_n(w)|^2\exp(-2n(\varphi_n(w)+\varphi_n(z)))\leq \det(dd^c\varphi_n(w))\det(dd^c\varphi_n(z))\exp(2\epsilon_n)\frac{C_1}{16c}e^{-\frac{T\sqrt{n}|z-w|}{12}} .$$
  \end{proof}
 \section{Asymptotic normality of random zeros} 
Let $\{g_j\}_{j\in\Bbb{N}}$ be a sequence of complex-valued measurable functions on a measure space $(X,\mu)$ such that
$$\sum_{j}|g_j(x)|^2=1\ \text{for every}\ x\in X.$$ 
A \textit{complex Gaussian process} is a complex-valued random function of the form
$$G(x)=\sum_{j}c_jg_j(x)$$ where $c_j$ are i.i.d. complex Gaussian random variables of mean zero and variance one. We also define the correlation function $\rho:X\times X\to \Bbb{C}$ of $G(x)$ by
$$\rho(x,y):=\Bbb{E}\{G(x)\overline{G(y)}\}=\sum_{j}g_j(x)\overline{g_j(y)}.$$ Note that $|\rho(x,y)|\leq 1$ and $\rho(x,x)=1$ for every $x,y\in X.$

\par For a sequence of complex Gaussian processes $G_1,G_2,\dots$ and a bounded measurable function $\psi:X\to \Bbb{R}$ one can define
$$Z_n^{\psi}(G_n):=\int_X\log|G_n(x)|\psi(x)d\mu(x).$$ 
Next result provides sufficient conditions on the asymptotic normality of linear statistics of $Z_n$.
\begin{thm}\cite{STr}\label{ST}
Let $\rho_n(x,y)$ denote the correlation function for $G_n.$ Assume that
 \begin{equation}\label{c1}\lim_{n\to \infty}\sup_{x\in X}\int_X|\rho_n(x,y)|d\mu(y)=0\end{equation}
\begin{equation}\label{c2}
\liminf_{n\to \infty}\frac{\int_X\int_X|\rho_n(x,y)|^2\psi(x)\psi(y)d\mu(x)d\mu(y)}{\sup_{x\in X}\int_X|\rho_n(x,y)|d\mu(y)}>0.
\end{equation}
Then linear statistics
$$\frac{Z_n^{\psi}(G_n)-\Bbb{E} Z_n^{\psi}(G_n)}{\sqrt{Var[Z_n^{\psi}(G_n)]}}$$ converges in distribution to the (real) Gaussian $ \mathcal{N}(0,1)$ as $n\to \infty.$
\end{thm}
We remark that condition \ref{c1} ensures that $\lim_{n\to \infty}Var[Z_n^{\psi}(G_n)]=0$ (c.f \cite{STr,SZ3} see also \cite[Lemma 5.2]{B6} for more general distributions).
\begin{proof}[Proof of Theorem \ref{CLT}]
Note that $\mathcal{Z}^{\Phi}_n=\langle [Z_{f_n}],\Phi\rangle= \int_{\C}\log|f_n|dd^c\Phi$ where $dd^c\Phi=\psi(z) dV_m$ for some $\mathscr{C}^1$ function $\psi$ which has compact support in $\mathcal{B}_n$ for each $n.$ We will verify the conditions of Theorem \ref{ST} for the Gaussian processes 
$$\mathcal{Z}_n^{\psi}(f_n)=\int_X\log|f_n|\psi(z)dV_m$$ 
with $X=supp(\psi)$ and $\mu$ is restriction of $dV_m$ to $X.$
\par First, we show that (\ref{c1}) holds. Note that
$$|\rho_n(z,w)|=\frac{|K_n(z,w)|}{\sqrt{K_n(z,z)}\sqrt{K_n(w,w)}}.$$
We fix $z\in X$ and estimate 
\begin{eqnarray*}
\int_{X}|\rho_n(z,w)|dV_m(w)&=&\int_{\{w\in X:|z-w|<\frac{\log n}{\sqrt{n}}\}}\ + \ \int_{\{w\in X:{|z-w|\geq\frac{\log n}{\sqrt{n}}}\}}\\
&=:& I_n^1(z)+I_n^2(z)
\end{eqnarray*}
Using $|\rho_n(z,w)|\leq 1$ we see that $$I_n^1(z)\leq C(\frac{\log n}{\sqrt{n}})^{2m}$$ for every $z\in X.$
Moreover, by Theorem \ref{thm2} and by (\ref{help})
\begin{eqnarray*}
I_n^2(z)\leq \frac{CVol_m(X)}{n^T}(1-C_8\frac{(\log n)^{m-1}}{\sqrt{n}})^{-1}. 
\end{eqnarray*}  Hence, we conclude that (\ref{c1}) holds.

\par In order to verify condition (\ref{c2}) we break up the integral again into two pieces where $|z-w|<\frac{\log n}{\sqrt{n}}$ and $|z-w|\geq\frac{\log n}{\sqrt{n}}.$ It follows from Theorem \ref{thm2} that the latter integral decay rapidly to 0. Hence, it is enough to estimate 
\begin{eqnarray}\label{equ}
& &\liminf_{n\to \infty}\frac{\int_{X}(\int_{|v|<\log n}|\rho_n(z,z+\frac{v}{\sqrt{n}})|^2\psi(z+\frac{v}{\sqrt{n}})dV_m(v))\psi(z)dV_m(z)}{\sup_{z\in X}\int_{|v|<\log n}|\rho_n(z,z+\frac{v}{\sqrt{n}})|dV_m(v) }.
\end{eqnarray} 
On the other hand, it is an immediate consequence of Theorem \ref{thm3} and the identity
$$\langle u,v\rangle-\frac12|u|^2-\frac12|v|^2=-\frac12|u-v|^2+iIm(\langle u,v\rangle) $$ that as $n\to \infty$
\begin{equation}
\frac{n^{-2m}|K_n(z+\frac{u}{\sqrt{n}},z+\frac{v}{\sqrt{n}})|^2e^{-2n\varphi_n(z+\frac{u}{\sqrt{n}})}e^{-2n\varphi_n(z+\frac{v}{\sqrt{n}})}}{\det(dd^c(\varphi_n(z))^2\exp(-\sum_{j=1}^m\lambda_j^n|u_j-v_j|^2)}\to1
\end{equation} 
 where $\lambda_j^n$ are eigenvalues of $dd^c\varphi_n(z).$ Hence, by Theorem \ref{thm1} and the preceding argument for Theorem \ref{thm3} we obtain
\begin{equation}
\rho_n(z+\frac{u}{\sqrt{n}},z+\frac{v}{\sqrt{n}})=\exp(-\frac12\sum_{j=1}^m\lambda_j^n|u_j-v_j|^2)(1+\tau_n(u,v))
\end{equation}
where $|\tau_n(u,v)|\leq A|u-v|^2 \frac{(\log n)^3}{\sqrt{n}}$ for $|u|,|v|<\log n.$  Then the limit (\ref{equ}) can be estimated below by 
\begin{eqnarray}\label{son}
& & \liminf_{n\to \infty}\frac{\int_{X}(\int_{|v|<\log n}\exp(-\sum_{j=1}^m\lambda_j^n|v_j|^2)(1+\tau_n(0,v))\psi(z+\frac{v}{\sqrt{n}})dV_m(v))\psi(z)dV_m(z)}{\int_{|v|<\log n}\exp(-\frac12\sum_{j=1}^m\lambda_j^n|v_j|^2)(1+\tau_n(0,v))^{\frac12}dV_m(v)}.
%&=& 2^{-m} \int_{\C}(\psi(z))^2dV_m(z)>0
 \end{eqnarray}
Since $\psi$ is $\mathscr{C}^1$ we have $\psi(z+\frac{v}{\sqrt{n}})=\psi(z)+O(|v|/\sqrt{n}).$ Hence, (\ref{son}) is bounded below by
\begin{eqnarray*}
& &\liminf_{n\to \infty}\frac{\int_{X}(\int_{|v|<\log n}\exp(-\sum_{j=1}^m\lambda_j^n|v_j|^2)(1-C|v|\frac{(\log n)^3}{\sqrt{n}})(\psi(z)+O(v/\sqrt{n}))dV_m(v))\psi(z)dV_m(z)}{\int_{|v|<\log n}\exp(-\frac12\sum_{j=1}^m\lambda_j^n|v_j|^2)(1+C|v|\frac{(\log n)^3}{\sqrt{n}})^{\frac12}dV_m(v)}\\
&=& 2^{-m} \int_{\C}\psi^2dV_m>0.
\end{eqnarray*}
\end{proof}
\begin{comment}
  We end this note with an example where Theorem \ref{CLT} applies however, normalized zero measures of random polynomials do not converge.
 \begin{example}\label{example}
 We consider
 $$\varphi_n:\Bbb{C}\to\Bbb{R}$$
 \begin{equation*}
\varphi_n(z) = \left\{
\begin{array}{rl}
\frac{|z-1|^2}{8} & \text{if } n=2k\\
\frac{|z+1|^2}{16} & \text{if } n=2k+1
\end{array} \right.
\end{equation*}
Then it is easy to see that
 \begin{equation*}
\varphi^e_{2k}(z) = \left\{
\begin{array}{rl}
\frac{|z-1|^2}{8} & \text{if } |z-1|<2\\
\log|z-1|+\frac12-\log2 & \text{if } |z-1|\geq 2
\end{array} \right.
\end{equation*} and 
 \begin{equation*}
\varphi_{2k+1}^e(z) = \left\{
\begin{array}{rl}
 \frac{|z+1|^2}{16}  & \text{if } |z+1|<2\sqrt{2}\\
\log|z+1|+\frac12-\frac32\log2 & \text{if } |z+1|\geq 2\sqrt{2}.
\end{array} \right.
\end{equation*}
Clearly, $\varphi_n$ verifies the hypotheses of Theorem \ref{CLT} and the closed unit disc $\overline{\Delta}$ is contained in the bulk $\mathcal{B}_n$ for every $n.$ Moreover, it follows from \cite[\S 6]{BloomL} that almost surely 
$$\frac{1}{2k}\log|f_{2k}|\to \varphi_{2k}^e$$
$$ \frac{1}{2k+1}\log|f_{2k+1}|\to \varphi_{2k+1}^e$$ as $k\to \infty$ in $L^1_{loc}(\Bbb{C}).$ In particular, their Laplacians also converge. However, $\partial\overline{\partial}\varphi_{2k}\not= \partial\overline{\partial} \varphi_{2k+1}$ on $\overline{\Delta}.$ 
 \end{example}
\end{comment}
  \bibliographystyle{alpha}
\bibliography{biblio}
\end{document}